\documentclass[11pt]{amsart}

\usepackage{graphicx}
\usepackage{pstricks, enumerate, pst-node, pst-text, pst-plot}
\usepackage{upgreek}
 \usepackage{epsfig}

\newtheorem{theorem}{Theorem}
\newtheorem{lemma}[theorem]{Lemma}
\newtheorem{proposition}[theorem]{Proposition}
\newtheorem{corollary}[theorem]{Corollary}

\theoremstyle{remark}
\newtheorem{remark}[theorem]{Remark}

\theoremstyle{definition}
\newtheorem{definition}[theorem]{Definition}

\numberwithin{equation}{section}

 \newcommand{\mc}{\mathcal}

 \newcommand{\E}{\mc{E}}
 \newcommand{\li}{\mc{L}}
 \newcommand{\q}{\mc{Q}}

 \newcommand{\M}{\mc{M}}

 \newcommand{\U}{\mc{U}}
 
 \newcommand{\Y}{\mc{Y}}
 \newcommand{\C}{\mathbb{C}}
 \newcommand{\R}{\mathbb{R}}

 \newcommand{\Kahler}{K\"{a}hler}
 \newcommand{\Rc}{\text{Rc}}
 \newcommand{\ab}{\alpha\bar{\beta}}
 \newcommand{\Czero}{\mathbb{C}^{2}\backslash\{0\}}
 
\newcommand{\Rm}{\text{\rm Rm}}

 \newcommand{\CP}{\mathbb{CP}}
 
 \newcommand{\al}{\alpha}
 \newcommand{\be}{\beta}

 \newcommand{\dalphabetabar}{\dfrac{\partial^2 }{\partial z^{\alpha} \partial \bar{z}^{\beta}}}

 \usepackage{epsfig}

\begin{document}

\title[Blow-up of 4-d Ricci flow singularities]{On the blow-up of  four dimensional \\ Ricci flow singularities}

\author[Davi M{\'{a}}ximo]{Davi M{\'{a}}ximo}
\address{Department of Mathematics, University of Texas at Austin, TX 78712, USA.}
\email{maximo@math.utexas.edu}



\date{\today}
\begin{abstract}
In this paper we prove a conjecture by Feldman-Ilmanen-Knopf in \cite{FIK} that the gradient shrinking soliton metric they cons\-tructed on the tautological line bundle over $\CP^1$ is the uniform limit of blow-ups of a type I Ricci flow singularity on a closed manifold. We use this result to show that limits of blow-ups of Ricci flow singularities on closed four dimensional manifolds do not necessarily have non-negative Ricci curvature. 
\end{abstract}
\maketitle

\section{Introduction}

Suppose $(M,g)$ is a closed Riemannian manifold.  One can evolve the me\-tric $g$ by
$\frac{\partial}{\partial t}g=-2\Rc(g),$
 a (weakly) parabolic system known as {\it Ricci flow}. Under certain curvature conditions on the initial $g$, it is possible to prove nor\-ma\-lized convergence to a round metric  (see Hamilton \cite{H1}, Brendle-Schoen \cite{BS}, B\"{o}hm-Wilking \cite{BW}) and thus say a lot about the to\-po\-lo\-gy of $M$. But for a large set of initial metrics the flow will become singular in finite time before converging to any smooth limiting metric.

 This happens when the curvature blows up on certain regions of the ma\-ni\-fold and, indeed, the standard short-time existence result for Ricci flow says that if the flow becomes singular at some finite time $T<\infty$ then $\lim_{t\nearrow T}\max_{x\in M} | \Rm(x, t)| = \infty$. Actually, one can prove that this happens if, and only if, $\limsup_{t\nearrow T}\max_{x\in M} | \Rc(x, t)| = \infty$, see \v{S}e\v{s}um \cite{S}.

 In order to analyze these singularities, one follows the conventional wisdom of singular analysis from non-linear PDEs and does a  {\it blow-up} at the sin\-gu\-la\-ri\-ty using the scaling symmetry of the equation. Depending on how much compactness is at hand,  one can extract a {\it singularity model} from a sequence of such blow-ups, which will usually have a better geometry than the original Ricci flow. Moreover, if one has a good knowledge about the possible singularity models, then one is able to understand the structure of the singularity formation and see how to perform surgery while controlling the geometry and the topology of the manifold, thus arriving at the so-called {\it Ricci flow with surgeries}.

The above blow-up analysis in three dimensional Ricci flow has been proved quite successful  by the work of Hamilton ($e.g.$~see \cite{H2})  and Perelman (\cite{P1}, \cite{P2}) on the Poincar\'{e} and Geometrization conjectures,  and a lot of their theory carries on to higher dimensional settings, which is an area of active research.

Indeed, for an $n$-dimensional Ricci flow $g(t)$ on a maximal time interval $[0,T)$ with $T<\infty$, it follows from Hamilton-Perelman's theory that one can choose a sequence of points $p_i\in M$ and  times $t_i\nearrow T$ with
$$\lambda_i=|\Rm|(p_i,t_i)= \sup_{x\in M,\, t\leq t_i} |\Rm|(x,t)\longrightarrow \infty$$
such that the rescaled flows
$$g_i(t)=\lambda_i g\left(t_i + \dfrac{t}{\lambda_i}\right)$$
will converge (in a suitable sense and up to subsequence) to a complete Ricci flow $(N,g_\infty(t))$, which one calls  {\it singularity model}. Moreover, if the singularity is of {\it type I}, i.e., the curvature blows up like
$$\limsup_{t\nearrow T}\max_M |\Rm|(T-t)<\infty,$$
by the work of Enders-M\"{u}ller-Topping in \cite{EMT} the above limit $(N,g_\infty(t))$ will be a non-flat  {\it gra\-dient shrinking soliton}, that is a self-similar Ricci flow, where the metric  $g_\infty$ evolves only by scaling and diffeomorphism. This means there will exist a time dependent function $f$ defined on $N$ such that one can obtain the Ricci flow $g_\infty(t)$ from pullbacks of an initial metric $g_\infty(0)$ by diffeomorphisms $\phi_t$ of $N$ generated by $\frac{1}{T-t}\nabla f$, that is,
$$g_\infty(t)= (T-t)\phi_t^{\ast}g_\infty(0),$$
and one can check this is the case whenever the soliton equation
\begin{equation}\label{eqsoliton}
\Rc({g_\infty}) +\nabla\nabla f=\dfrac{1}{2}g_\infty
\end{equation}
is satisfied at some point in time (see Cao \cite{Cao2} for a survey on metrics satisfying the above equation).

In dimension three, the {\it Hamilton-Ivey pinching} estimate (see \cite{H2}, \cite{I}) roughly states that if the flow has a region with very negative sectional curvature, then the most positive sectional curvature is much larger still. This implies that limits of blow-ups of three dimensional Ricci flows will have non-negative sectional curvature and thus drastically constrains the singularities that can appear, making three dimensional Ricci flow with surgeries plausible. In higher dimensions, one has that such limits will have non-negative scalar curvature by the work of Chen \cite{Ch1}, but this type of estimate is lacking for more useful curvature conditions.

In this paper we prove that certain such estimates {\it cannot} exist for Ricci curvature in dimension four, i.e., that limits of blow-ups of four-dimensional Ricci flows do not necessarily have non-negative Ricci curvature. We achieve this by solving a question left by Feldman-Ilmanen-Knopf in \cite{FIK},  which we state soon.

Let $M$ be $\CP^2$ blown-up at one point and $L$ be $\C^2$ blown-up at zero. We invite the reader not familiar with these spaces to see Appendix \ref{ap6.1} before reading what follows.

We will think of $M$ as $\Czero$ with one $\CP^{1}$  glued at $0$ (the section $\Sigma_0$) and another at $\infty$ (the section $\Sigma_\infty$) and of $L$ as $\Czero$ with a $\CP^{1}$  glued only at $0$. Both $M$ and $L$ are line bundles over $\CP^1$, $M$ with line $\CP^1$ and $L$ with line $\C$. One can consider {\it \Kahler~metrics} on these manifolds and evolve them by Ricci flow.

Indeed, let $g(t)$ be a one-parameter family of Riemannian metrics on the manifold $M$ evolving by Ricci flow\footnote{In the context of \Kahler~geometry, Ricci flow appears in the literature as $\partial_t g=-\Rc(g)$ instead of the usual $\partial_t g =-2\Rc(g)$. This only changes things by scaling.}
\begin{equation}\label{eqKRF}
\dfrac{\partial}{\partial t} g = -\Rc (g).
\end{equation}
Assume further that $g(0)$ is \Kahler. It is a noted fact that $g(t)$ will remain \Kahler~ with regard to the same complex structure and the flow is thus called \Kahler-Ricci flow, see $e.g.$~Cao \cite{Cao11}. The \Kahler~ class $[\omega(t)]$ of the metric $g(t)$ will evolve by
\begin{equation}
\partial_t[\omega(t)]=-[\Rc(\omega)]=-c_1(M),
\end{equation}
\noindent where $c_1(M)$ is the first Chern class of the complex surface $M$. In particular,
\begin{equation}\label{eq:KRF0}
[\omega(t)]=[\omega(0)]- t c_1(M).
\end{equation}
\indent Moreover, on $M$, the cohomology classes of the divisors $[\Sigma_0]$ and $[\Sigma_\infty]$  span $H^{1,1}(M;\R)$ so any \Kahler~ class $[\omega]$ can be written uniquely as
$$[\omega]=b[\Sigma_\infty]-a[\Sigma_0]$$
\noindent for constants $0<a<b$,  and the first Chern class satisfies $ c_1(M)=[\Sigma_0]+3[\Sigma_\infty]$. This and equation (\ref{eq:KRF0}) give
\begin{equation}\label{kahlerclass}
[\omega(t)]=b(t) [\Sigma_\infty] - a(t)[\Sigma_0]
\end{equation}
\noindent for $a(t)=a(0)-t$ and $b(t)=b(0)-3t$. Thus, if initially $b(0)>3a(0)$, then $a(t)\rightarrow 0$ as $t\nearrow T=a(0)$ and the class $[\omega(T)]$ will not be \Kahler.  This will mean that the $\CP^1$ of the section $[\Sigma_0]$ has collapsed to a point and thus Ricci flow must have become singular at a time no later than $t=T$.

In \cite{FIK}, Feldman-Ilmanen-Knopf conjectured that indeed, at least for $U(2)$-invariant metrics, if $b(0)>3a(0)$ then a type I singu\-la\-rity  will develop along $\Sigma_0$ precisely at time $t=T$ and the blow-up limit of such singularity is the gradient shrinking soliton they have constructed on $L$, the {\it FIK soliton}.

Since their work,  a lot of investigation has been done on \Kahler-Ricci flow of general \Kahler~manifolds. Of relevance to this paper are Tian-Zhang \cite{TZ}, Song-Weinkove \cite{SW}, and the more recent Song \cite{So} .  When restricted to the  \Kahler-Ricci flow of $U(2)$-invariant metrics of $M$ as above, \cite{TZ} gives the singular time to be exactly $T=a(0)$, and in \cite{SW} the authors prove, among other things,  that the singularity at $t=T$ will develop only along $\Sigma_0$, with $g(T)$ being a smooth metric on $M\backslash\Sigma_0$. Finally, Song \cite{So} proved that such a singularity is type I, and using the compactness at hand, he argued that limits of blow ups of the flow will subconverge in the Cheeger-Gromov-Hamilton sense to a complete non-flat gradient shrinking \Kahler-Ricci soliton on a manifold homeomorphic to $\C^2$ blown up at one point. Moreover, the isometry group of this soliton contains the unitary group $U(n)$; but since the complex structure might jump in the limit, one is not able to argue using Cheeger-Gromov-Hamilton convergence that this soliton is in fact the FIK soliton.

  In this paper we complete the proof of the Feldman-Ilmanen-Knopf  conjecture for a large set of initial metrics:


 \begin{theorem}\label{thm:main}
 Let $g(t)$ be metrics on $M$ evolving by \Kahler-Ricci flow \eqref{eqKRF}. For a large open set\,\footnote{See Definition \ref{defc} in Section 3 for a precise statement on the set of initial metrics.} of $U(2)$-invariant initial me\-trics $g(0)$ belonging to the \Kahler~ class $b(0)[\Sigma_\infty]-a(0)[\Sigma_0]$ with $b(0)>3a(0)$, one has the following:
 \begin{itemize}
 \item[(i)] the flow is smooth until it becomes singular at time $T=a(0)$,
 \item[(ii)] at t=T the flow develops a type I singularity in the region $\Sigma_0$ and $g(T)$ is a smooth Riemannian metric on $M\backslash\Sigma_0$,
 \item[(iii)] parabolic dilations of $g(t)$ converge uniformly  on any parabolic neighborhood\,\footnote{This notion is made precise in Section 3 too. See also Remark \ref{remark3}.}  of the singular set $\Sigma_0$ to the evolution of the FIK soliton.
 \end{itemize}
\end{theorem}

As we have pointed out, part (i) follows from a now standard general result of \cite{TZ} and part (ii) is proved by putting together the results in \cite{SW} and \cite{So}. Our proof of part (iii) is  based on comparison principle techniques applied to the evolving metric potentials, and thus gives  a \Kahler~limit with respect to the same original complex structure. We first prove convergence to  the FIK metric in $C^{0,1}$ topology without making any type I blow up assumptions. To prove higher regularity without making further restrictions on the class of initial data, we then use the type I blow up for the scalar curvature proved in \cite{So}.

We remark that comparison principle techniques have also been used on a large body of work on yet another type of Ricci flow singularity, {\it neckpinches}, see Angenent-Knopf \cite{AK1}, \cite{AK2},  Angenent-Caputo-Knopf \cite{ACK}, Angenent-Isenberg-Knopf \cite{AIK}, and Gu-Zhu \cite{GZ}.
 
For the reader interested  in more results concerning singularity analysis of the \Kahler-Ricci flow we suggest to look also at Song-Weikove \cite{SW2}, \cite{SW3}, Song-Tian \cite{ST}, Fong \cite{Fo}, and the references therein.  In the latter article, Fong studies singularity formation for the case $b(0)<3a(0)$. For general \Kahler-Ricci flow lecture notes see Song-Weinkove \cite{SW4}.

Theorem \ref{thm:main} has the following two consequences. Since the FIK soliton has Ricci curvature of mixed sign near $\Sigma_0$:

\begin{theorem}\label{thm:2}
Limits of blow-ups of Ricci flow singularities on closed four dimensional manifolds do not necessarily have non-negative Ricci curvature. 
\end{theorem}

Moreover, after constructing a metric  on $M$ with strictly positive Ricci curvature that satisfies the conditions of Theorem \ref{thm:main}, we will have:

\begin{corollary}\label{cor}
Positive Ricci curvature is not preserved by Ricci flow in four dimensions or higher.
\end{corollary}

As mentioned before, the above results shows a contrast  between Ricci flow in dimensions three and four. Moreover, Corollary \ref{cor} is related to a previous result of the author \cite{M}, and also a result on lower bounds for Ricci curvature under the \Kahler-Ricci flow by Zhang \cite{Zh}, which provides other examples that imply the same result stated in the corollary.
\\\\
\noindent{\bf Acknowledgments.} This work is part of my PhD thesis at the University of Texas at Austin. It could not have been done without the support and mentorship of my advisor Dan Knopf, to whom I am deeply grateful. I also wish to warmly thank Jian Song and Ben Weinkove for hepful  comments on an earlier version of this work, and to acknowledge the NSF for its support.


\section{$U(2)$-invariant \Kahler~metrics}


In this section we consider rotationally symmetric \Kahler~ metrics on $\Czero$ to derive \Kahler~metrics on any given \Kahler~ class  of the complex surface $M$, following an ansatz introduced by Calabi \cite[Section 3]{C}.\\
\indent Let $g$ be a $U(2)$-invariant \Kahler~ metric on $\Czero$, the latter with complex coordinates $z=(z^1,z^2)$. Define $u=|z^1|^2$, $v=|z^2|^2$, and $w=u+v$.\\
\indent Since $g$ is a \Kahler~ metric and the second de Rham cohomology group $H^2(\Czero)=0$, by the $\partial\bar{\partial}$-lemma one can find a global real smooth function $P:\Czero\longrightarrow\R$ such that
\begin{equation}\label{eq: kahler potential}
g_{\ab}=\dalphabetabar P.\,
\end{equation}
\indent The further assumption of $g$ being rotationally symmetric allows us to write $P=P(r)$, where $r=\log w$.\footnote{Depending on the purpose of the computation, one coordinate might be preferable than the other and we will use both $r$ and $w$ in the rest of the paper.} We then set $\varphi(r)=P_r(r)$ (we use a subscript for the derivative since later $P$ will be regarded as a function of time as well) and compute from (\ref{eq: kahler potential})
\begin{equation}\label{eq:g}
g=[e^{-r}\varphi\delta_{\alpha\beta} + e^{-2r}(\varphi_r-\varphi)\bar{z}^{\al}z^{\beta}]dz^{\al}d\bar{z}^{\beta},
\end{equation}
\noindent and:\footnote{The matrix $(g)$ is actually a $4\times 4$ matrix: $(g)=\left(
\begin{matrix}
A & 0\\
0& A\\
\end{matrix}
\right)$, where $A=
\left(\begin{matrix}
g_{1\bar{1}} & g_{1\bar{2}}\\
g_{2\bar{1}} & g_{2\bar{2}}\\
\end{matrix}
\right)$.}
\begin{equation*}
\left(
\begin{matrix}
g_{1\bar{1}} & g_{1\bar{2}}\\
g_{2\bar{1}} & g_{2\bar{2}}\\
\end{matrix}
\right)=\dfrac{1}{w^2}\left(
\begin{matrix}
v\varphi+u\varphi_r & (\varphi_r-\varphi)\bar{z}^{1}z^{2}\\
(\varphi_r-\varphi)z^{1}\bar{z}^{2} & u\varphi+v\varphi_r\\
\end{matrix}
\right).
\end{equation*}
\indent Because $\det (g_{\ab})= e^{-2r}\varphi\varphi_r$, one  can quickly note that a potential $P$ on $\Czero$ gives rise to a \Kahler~ metric as in (\ref{eq: kahler potential}) if, and only if,
\begin{equation}\label{eq:Kahler condition}
\varphi>0\,\, \text{and}\,\, \varphi_r>0 .
\end{equation}
\indent Given a metric as above, Calabi's Lemma \cite[Section 3]{C} tells us that $g$ will extend to a smooth \Kahler~ metric on the complex surface $M$ if $\varphi$ satisfies the following asymptotic properties $-$ henceforth called {\it Calabi's conditions}:
\begin{enumerate}
\item[(i)] There exists positive constants $a_0$ and $a_1$ such that $\varphi$ has the expansion
 \begin{equation}\label{calabi:condition0}
\varphi(r)=a_0+a_1w+a_2w^2+{O}(|w|^3)
\end{equation}
\noindent as $r\rightarrow -\infty$;
\item[(ii)] There exists a positive constant $b_0$ and a negative constant $b_1$ such that $\varphi$ has the expansion
\begin{equation}\label{calabi:condition1}
\varphi(r)=b_0+b_1w^{-1}+b_2w^{-2}+{O}(|w|^{-3})
\end{equation}
\noindent as $r\rightarrow \infty$.
\end{enumerate}
\begin{remark}
We note that $\varphi_r>0$ for $r$ finite, but $\varphi_r=0$ for $r=\pm\infty$.
\end{remark}

\indent Summarizing the above:

\begin{lemma}[Calabi, \cite{C}]
Any potential $\varphi$ satisfying conditions $(\ref{eq:Kahler condition})$, $(\ref{calabi:condition0})$ and $(\ref{calabi:condition1})$ will give rise to a $U(2)$-invariant \Kahler~metric on $M$.  Moreover, this me\-tric will belong in \Kahler~class $b_0[\Sigma_\infty]-a_0[\Sigma_0]$, and satisfy $|\Sigma_0|=\pi a_0$ and $|\Sigma_\infty|=\pi b_0$.
\end{lemma}

\subsection{Curvature Terms}

On \Kahler\ manifolds, the Ricci tensor is given locally by
\begin{equation*}
\Rc_{\ab}=-\dalphabetabar\log\det g.
\end{equation*}
\noindent In particular, for $g$ as in (\ref{eq:g}), one has globally
\begin{equation}\label{eq:Ricci}
\Rc_{\ab}=e^{-r}\psi\delta_{\alpha\beta}+e^{-2r}(\psi_r-\psi)\bar{z}^{\al}z^{\be}
\end{equation}
\noindent where $\psi=- \partial_r(\log\det g)=  2-\frac{\varphi_r}{\varphi}-\frac{\varphi_{rr}}{\varphi_r}$. From equations $(\ref{eq:g})$ and $(\ref{eq:Ricci})$, we compute the eigenvalues of the Ricci curvature endomorphism\footnote{The map $\Rc: TM \longrightarrow TM$ obtained by raising one index.}
\begin{equation}\label{eigen}
\begin{array}{rcccl}
\lambda_1&=&\frac{\psi}{\varphi}&\text{with eigenvector}& U=\bar{z}^2\frac{\partial}{\partial z^1}+\bar{z}^1\frac{\partial}{\partial z^1}\\
\lambda_2&=&\frac{\psi_r}{\varphi_r}&\text{with eigenvector}& V=z^1\frac{\partial}{\partial z^2}+z^2\dfrac{\partial}{\partial z^2}.
\end{array}
\end{equation}
and, in particular, the scalar curvature
\begin{equation}\label{scalar}
R(r,t)= \frac{2}{\varphi} \left(2-\frac{\varphi_r}{\varphi}-\frac{\varphi_{rr}}{\varphi_r}\right) +\frac{2}{\varphi_r} \left[\left(-\frac{\varphi_r}{\varphi}\right)_r+\left(-\frac{\varphi_{rr}}{\varphi_r}\right)_r\right].
\end{equation}
Using Calabi's conditions we find for $r$ near $-\infty$
\begin{equation}\label{eigen1}
\begin{array}{rcl}
\lambda_1 & =& \frac{1}{a_0} +O(e^{r})\\
\lambda_2 &=&- \frac{1}{a_0}-\frac{2a_2}{a_1^2} +O(e^{r}),
\end{array}
\end{equation}
and for $r$ near $+\infty$
\begin{equation}\label{eigen2}
\begin{array}{rcl}
\lambda_1 & = & \frac{3}{b_0} +O(e^{-r})\\
\lambda_2 & = &\frac{1}{b_0}+\frac{2b_2}{b_1^2} +O(e^{-r}).
\end{array}
\end{equation}
Moreover, for the Riemann curvature, a direct computation shows
\begin{eqnarray}\label{riemann}
R_{\al\bar{\beta}\gamma\bar{\delta} } &=& e^{-4r}\left[-\varphi_{rrr}+4\varphi_{rr} -2\varphi_r +2\varphi -4 \frac{\varphi_r^2}{\varphi} + \frac{\varphi^2_{rr}}{\varphi_r}\right] \bar{z}^\al z^\beta\bar{z}^\gamma z^\delta\notag\\
&&+ e^{-3r}\left[\varphi_r- \varphi_{rr}- \varphi +\frac{\varphi_r^2}{\varphi}\right] (\bar{z}^\al z^\beta \delta^{\gamma\delta}+ \bar{z}^\al \delta^{\be\gamma}  z^\delta \notag\\
&&+ \delta^{\al\be}\bar{z}^\gamma z^\delta+\delta_{\al\delta}z^\be\bar{z}^{\gamma})  + e^{-2r}\left[-\varphi_r+\varphi \right] (\delta_{\al\be}\delta_{\gamma\delta}+ \delta_{\al\delta}\delta_{\be\gamma}).
\end{eqnarray}

\subsection{The effect of Ricci flow}

\indent  From equations $(\ref{eq:g})$ and $(\ref{eq:Ricci})$, one can see that $g(t)$ evolves by Ricci flow $\partial_t g=-\Rc(g)$ if, and only if, $\varphi$ evolves by $\varphi_t=-\psi$, that is,
\begin{equation}\label{eq:varphi}
\varphi_t=\frac{\varphi_{rr}}{\varphi_r}+\frac{\varphi_r}{\varphi}-2.
\end{equation}

\begin{remark} Equation $(\ref{eq:varphi})$ looks alarming from the PDE point of view, as it might de\-ge\-ne\-rate. On the other hand, we recall that for potentials $\varphi$ that yield \Kahler~metrics on $\Czero$ we have $\varphi_r>0$ and this condition is preserved, so $(\ref{eq:varphi})$ is parabolic.
\end{remark}

The \Kahler~ class will evolve as in \eqref{kahlerclass} and
\begin{eqnarray*}
\varphi(-\infty,t)=a(t)=a_0-t,\,\,\,\varphi(+\infty,t)=b(t)=b_0-3t,
\end{eqnarray*}
so the flow will become singular no later than $t=a(0)$.  In fact, a general result in \Kahler-Ricci flow,  see \cite{TZ}, says that the \Kahler-Ricci flow exists and is smooth up until the first time $t=T$ where the \Kahler~ class $[w(T)]$ ceases to be \Kahler.  In our case, because $b(0)>3a(0)$, this happens precisely at $T=a(0)$. Furthermore,  since $|\Sigma_0|=\pi a(t)$, the section $\Sigma_0$ will vanish when $t=T$ and, as it turns out,  $g(T)$ is still smooth in $M\backslash\Sigma_0$, see \cite{SW}.

The following scale invariant estimate is an immediate consequence of the maximum principle and will be useful later.

\begin{lemma}\label{phirphi}
Let $F=\frac{\varphi_r}{\varphi}$. Then, $0\leq {F(r,t)\leq \max\{\max F(\cdot, 0), 1\}.}$
\end{lemma}
\begin{proof}
Because of \eqref{eq:Kahler condition}, $F\geq0$. To prove the upper bound, let  $r_0$ be a local spatial maximum of $F$, at which we must have
$$F_{rr}(r_0)\leq 0\,\,\text{and}\,\,F_r(r_0)=0.$$
We then compute the evolution of $F$
$$F_t=\dfrac{F_{rr}}{\varphi_r}+\dfrac{F_r}{\varphi}\left(1-\dfrac{F_r}{F^2}
\right)+2\dfrac{F}{\varphi}(1-F),$$
and thus note that 
$$F_t(r_0)=\dfrac{1}{\varphi}\left[\dfrac{F_{rr}}{F}+2F(1-F)\right]\leq
\dfrac{2}{\varphi}F(1-F).$$ 
\noindent So,  at a local spatial maximum where $F(r_0)>1$ (respectively $\geq 1$), $F(x,t)$ is decreasing (respectively non-increasing) in time, and the lemma follows since $F\equiv 0$ in the boundary points $r=-\infty,+\infty$.
\end{proof}
\subsection{Soliton metrics on $L$ and $M$}
Following \cite{FIK}, consider a $U(2)$-invariant gra\-dient shrinking soliton metric on $L$ or $M$, normalized such that $|\Sigma_0|=\pi$. Then it must have a potential $\varphi$ satisfying
\begin{equation}\label{solitoneq}
\dfrac{\varphi_{rr}}{\varphi_r}+\dfrac{\varphi_{r}}{\varphi}-C\varphi_r+\varphi-2=0,
\end{equation}
for some constant $C$ and with $\lim \varphi=1$ as $r\rightarrow -\infty$ and another asymptotic  condition as $r\rightarrow +\infty$, depending on whether one is looking for a soliton on $M$ or $L$.

In the case one is looking on $M$, by the independent work of Koiso in \cite{Ko} and Cao in \cite{Cao1},  $C$ must be a constant between $\frac{1}{2}$ and $1$ and there will exist only one such potential, modulo translations in $r$. This soliton has positive Ricci curvature, as Cao observed in \cite{Cao1}, and satisfies $a_0=1$, $ b_0=3$.

For a soliton on $L$, Feldman-Ilmanen-Knopf proved that the $C$ must be $\sqrt{2}$ and that there exists only one such potential, again modulo translations in $r$. For our purposes, it is relevant to note that this soliton has Ricci curvature of mixed sign: $\lambda_2<0$ for $r$ near $-\infty$. In fact, using the soliton equation \eqref{solitoneq} and the expression \eqref{eigen} for the eigenvalue one finds that
$\lambda_2=1-\sqrt{2}\frac{\varphi_{rr}}{\varphi_r}$
and since $\lim\limits_{r\rightarrow-\infty}\frac{\varphi_{rr}}{\varphi_r}=1$, we have that $\lambda_2$ is negative near the section $\Sigma_0$.


\section{Dilation Variables, Type I blow-up, \\ and convergence modulo Diffeo\-mor\-phisms}


\subsection{Dilation Variables}
What we know so far about the singularity formation is that it occurs along the section $\Sigma_0$, which shrinks to a point by \cite{SW}, and that it is type I \cite{So}. Hence it will be useful to use {\it parabolically dilated variables} that allow us to {\it zoom in on $\Sigma_0$ in a type I fashion}.\\
\indent Given an evolution $\varphi(r,t)$ with singular time $T$ as before, we define the {\it dilated time} variable $\tau=-\log(T-t)$, the {\it dilated spatial} variable $\rho = r +\tau$ (which correspond to complex coordinates $\zeta=e^{\tau/2}z$) and set
 \begin{equation}\label{dilated}
  \phi(\rho,\tau)=e^{\tau}\varphi(r(\rho,\tau),t(\tau)).
  \end{equation}
 \indent The function $\phi(\rho,\tau)$ evolves by
 \begin{equation}\label{eq:evoldil}
 \phi_\tau=\dfrac{\phi_{\rho\rho}}{\phi_\rho}+\dfrac{\phi_\rho}{\phi}-\phi_\rho+\phi-2
 \end{equation}
  and is a \Kahler~potential on $M$ in complex coordinates $\zeta$. To see that,  we note that it must satisfy $\phi,\phi_\rho>0$ for all $\tau$, and the Calabi conditions $(\ref{calabi:condition0})$ and $(\ref{calabi:condition1})$ --- or, alternatively, that $\phi$ represents the metric $\bar{g}(\tau)$ on $M$ equivalent to $g(t)$ scaled by $e^\tau=\frac{1}{T-t}$ and pulled back by the diffeomorphism $z \longrightarrow e^{-\tau/2}z$. \\
\indent Now that $\rho,\tau$ are defined, we explain what we mean by zoom in $\Sigma_0$ in a type I fashion, first on the level of the potential and then on the level of the metric.

For fixed $\rho$, we let $t\nearrow T$ and thus have $\tau\nearrow +\infty$ and $(r(\rho,\tau),t(\tau))\rightarrow (-\infty,T)$. The dilation
$$\phi(\rho,\tau) = \frac{1}{T-t} \varphi(\rho-\tau, t)$$
\noindent is then a type I (note the factor $\frac{1}{T-t}$) zoom in on how $\varphi$ is going to zero along $\Sigma_0$, since $\varphi(\rho-\tau,t)\rightarrow \varphi (-\infty,T)=0$.

Geometrically, the metric $\bar{g}(\tau)$ given by $\phi(\rho,\tau)$ is just  $\frac{1}{T-t}g(t)$ modified by diffeomorphism, and if  $g(t)$ has a type I singularity at $t=T$, $\bar{g}(\tau)$ has bounded curvature as $\tau\nearrow +\infty$. Moreover,
\begin{eqnarray*}
|\Sigma_0|_{\bar{g}(\tau)} &=&\pi\\
|\Sigma_\infty|_{\bar{g}(\tau)} &=&[(b_0-3a_0)e^\tau+3]\pi,
\end{eqnarray*}
\noindent which indicates that as $\tau\nearrow+\infty$, $\Sigma_\infty$ is being blown away and suggests that  the metrics $\bar{g}(\tau)$ on $M$ are becoming more and more like metrics on $L$ where $|\Sigma_0|=\pi$, as one recalls from Appendix \ref{ap6.1}.

Our goal is to prove that indeed the limit of these metrics will be mo\-de\-led by the evolution of the FIK soliton metric on $L$ constructed in \cite{FIK}. But there are major difficulties that  do not allow us to prove directly  such convergence by working with $\rho,\tau$. One comes from the fact that the FIK potential is not stationary in these coordinates; in fact, the soliton is still moving by diffeomorphisms, or more precisely, it is translating in the $\rho$ variable.
The other is the fact that there actually exist a whole family of FIK potentials, generated by the translations in $r$. These difficulties lead us to the approach presented in the next section.

\subsection{Equations in $\phi,\tau$ variables}

Let $\tau=-\log(T-t)$ and $\rho=r+\tau$ be the dilation variables introduced above. Because $\phi_\rho>0$ along the flow, we can actually write $\rho$ as a function of $\phi$ at any fixed time $\tau$ and thus consider the function
\begin{equation}
y(\phi,\tau)= \phi_\rho(\rho,\tau),
\end{equation}
which, for any given $\tau$, is defined on the inverval $[1,(b_0-3a_0)e^\tau+3]$, satisfying
$y(1)=y((b_0-3a_0)e^\tau+3)=0$ and $y(\phi)$ positive otherwise.

We next find the evolution equation of $y(\phi,\tau)$. First we note:  since $\phi_\rho=y$, we have $\phi_{\rho\rho}=y_\rho=y_\phi y$ and $\phi_{\rho\rho\rho}= y_{\phi\phi}y^2+y_\phi^2y$. And we then compute $\partial_\tau\big|_\rho y$, which in our notation means {\it the derivative of $y$ with respect to $\tau$ while fixing $\rho$}.
\begin{eqnarray*}
\partial_\tau\big|_\rho y&=& \phi_{\rho\tau} \\
&=& \frac{\phi_{\rho\rho\rho}}{\phi_{\rho}}- \left(\frac{\phi_{\rho\rho}}{\phi_\rho}\right)^2 + \frac{\phi_{\rho\rho}}{\phi}- \left(\frac{\phi_{\rho}}{\phi}\right)^2 +\phi_{\rho}-\phi_{\rho\rho}\\
&=& y_{\phi\phi} y + \frac{y_\phi y}{\phi}-\left(\frac{y}{\phi}\right)^2+y-y_\phi y.
\end{eqnarray*}
Finally, since $\partial_\tau\big|_\phi y = \partial_\tau\big|_\rho y - y_\phi\phi_\tau $, we have
\begin{eqnarray}\label{evol}
\partial_\tau\big|_\phi y&=& y_{\phi\phi} y + \frac{y_\phi y}{\phi}-\left(\frac{y}{\phi}\right)^2+y-y_\phi y - y_\phi \phi_\tau \notag \\
&=& y_{\phi\phi} y + (2-\phi-y_\phi)y_\phi + y\left( 1- \frac{y}{\phi^2}\right).
\end{eqnarray}

\begin{remark}\label{boundy}
Because of Lemma \ref{phirphi}, $\frac{y}{\phi}$ is uniformly bounded in time.
\end{remark}

The advantages of these variables are two-fold. First of all, they do not see translations in the $\rho$ variable and thus the whole family of FIK potentials are represented by just one  stationary potential $\mathcal{Y}=\mathcal{Y}(\phi^{\text{FIK}})$. Secondly, the non-linearities of \eqref{evol} are mild when compared with \eqref{eq:evoldil}, and this allow us to develop a barrier method based on the comparison principle in the next section.

\begin{remark}
In fact, one can check that all the FIK potentials satisfy the same equation when written using the coordinates as above
$$\Y(\phi)=\phi_\rho=\dfrac{ \phi( \phi-2)+  \sqrt{2}  (\phi-1) +1}{ \sqrt{2} \phi},$$
and also that $\partial_\tau\big|_\phi \mathcal{Y}=0$.
\end{remark}

\subsection{Comparison and convergence modulo diffeo\-mor\-phisms for a large class of potentials}
In this section we use techniques based on the comparison principle to prove an important step  towards Theo\-rem \ref{thm:main}, which is convergence modulo diffeo\-mor\-phisms for a large class of potentials.
\\\\
\noindent{\bf Notation.} {\it In what follows we will write just $\partial_\tau$ as short for $\partial_\tau\big|_\phi$. In par\-ticu\-lar, $\partial_\tau$ and $\partial_\phi$ are commuting derivatives.}
\\\\
\indent The first thing we prove is that for any initial data satisfying the conditions of Theorem  \ref{thm:main}, one always has a $C^1$ bound.

\begin{lemma}\label{boundyphi}
For any data $y(\phi,\tau)$  coming from Ricci flow as in Theorem \ref{thm:main},  the derivative $y_\phi$ is uniformly bounded in time.
\end{lemma}
\begin{proof}
Since at the boundary $y_\phi$ takes values $1$ or $-1$, we only need to deal with interior spatial maxima and minima. We write the evolution equation for $y_\phi$:
\begin{equation}\label{eqyphi}
y_{\phi\tau}= yy_{\phi\phi\phi} + [2-\phi-y_\phi]y_{\phi\phi} -2\frac{yy_\phi}{\phi^2} +\frac{2y^2}{\phi^3}.
\end{equation}
If  $y_\phi$ has a negative local spatial minimum at  $\phi_0$, then $y_{\phi\phi}(\phi_0)=0$,  $y_{\phi\phi\phi}(\phi_0)\geq 0$, and thus
$$y_{\phi\tau}(\phi_0)= yy_{\phi\phi\phi}(\phi_0)  -2\frac{yy_\phi}{\phi^2}(\phi_0) +\frac{2y^2}{\phi^3}(\phi_0)>0,$$
so $y_\phi$ is uniformly bounded from below.

Finally, let $\phi_0$ be a local spatial maximum of $y_\phi$, so that $y_{\phi\phi}(\phi_0)=0$ and  $y_{\phi\phi\phi}(\phi_0)\leq 0$. By Remark \ref{boundy}, $\frac{y}{\phi}<C$ for some constant $C$ independent of time. Suppose $y_\phi(\phi_0)>C$, then
\begin{eqnarray*}
y_{\phi\tau}(\phi_0)&=& yy_{\phi\phi\phi}(\phi_0)  -2\frac{yy_\phi}{\phi^2}(\phi_0) +\frac{2y^2}{\phi^3}(\phi_0)\\
&\leq&2\frac{y}{\phi^2}\left[ -y_\phi + \frac{y}{\phi} \right] <0.
\end{eqnarray*}

We hence conclude $y_\phi$ is uniformly bounded from above too and the lemma is proved.
\end{proof}

\begin{remark}\label{boundedcoef}
The above lemma, together with Remark \ref{boundy}, says that  the evolution \eqref{evol} has bounded coefficients on any compact interval $[1,\phi_0]$.
\end{remark}
\begin{remark}
Because $y_\phi=\frac{\phi_{\rho\rho}}{\phi_\rho}$, the lemma will be useful to prove $C^2$ bounds for $\phi$.
\end{remark}

We next construct  upper and lower barriers that for  ``most'' initial data will trap the solution to evolution \eqref{evol} and squeeze it to the FIK potential $\Y(\phi)=\frac{ \phi( \phi-2)+  \sqrt{2}  (\phi-1) +1}{ \sqrt{2} \phi}$.

\begin{definition}[\it Metrics in the class $\mc{C}$]\label{defc}
Let $\mc{C}$ be the class of all initial $U(2)$-invariant metrics  on $M$  belonging to the \Kahler~  class $b(0)[\Sigma_\infty]-a(0)[\Sigma_0]$ with $b(0)>3a(0)$, and such that, moreover, the parabolic blow up $\phi$ of the potential $\varphi$ satisfies, when writing $\phi_\rho$ in $\phi$ coordinates as before,
$$y(\phi,0)> \Y(\phi) - \frac{1}{5}\phi^2.$$
\end{definition}

Let us remark that  $\Y(\phi) - \frac{1}{5}\phi^2$ will be exactly the initial barrier that we will use for the evolution \eqref{evol}, thus we are restricting ourselves to potentials that are initially strictly above it. This barrier is mostly negative, but is positive and small (strictly less than 0.06) on a small neighborhood of $\phi=2$. This implies that a large family of initial data belongs to the class $\mc{C}$.

\begin{proposition}\label{conv1}
For any initial data $y(\phi,0)$ in the class $\mc{C}$, one has that $y(\phi,\tau)$ remains in the class $\mc{C}$ and converges uniformly on compact subsets to $\Y(\phi)$  as $\tau\nearrow\infty$.
\end{proposition}
\indent We note that the ``elliptic" operator in \eqref{evol}
$$\E [y] = y_{\phi\phi} y + (2-\phi-y_\phi)y_\phi + y\left( 1- \frac{y}{\phi^2}\right)$$
can be written as a linear plus a quadratic part
$$\E[y]=\li[y]+\q[y],$$
where $\li[y]=(2-\phi)y_\phi+y$ and $\q[y]=yy_{\phi\phi}-y_\phi^2-\frac{y^2}{\phi^2}$. The non-linearity of $\E[\cdot]$ is such that for given functions $y(\phi),s(\phi)$:
\begin{eqnarray*}
\E[y+s] &=& \li[y+s]+ \q[y+s]\\
&=&\li[y]+\li[s] + \q[y] + \q[s] + \M[y,s]\\
&=& \E[y]+ \E[s] + \M[y,s],
\end{eqnarray*}
where $\M[y,s]= sy_{\phi\phi}+ys_{\phi\phi} - 2y_\phi s_\phi -2\frac{ys}{\phi^2} $ is bilinear in $y$ and $s$. This mild non-linearity suggests the barrier approach.

\begin{proof}[Proof of Proposition \ref{conv1}]  Let
$$\Y(\phi)=\dfrac{ \phi( \phi-2)+  \sqrt{2}  (\phi-1) +1}{ \sqrt{2} \phi}$$
be the FIK potential and consider $s(\phi,\tau)=-\lambda(\tau) \phi^2$, where $\dot{\lambda}=-\delta\lambda$ for some $\delta \in\R$. We compute $(\partial_\tau - \E )(\Y+s)$ :
\begin{eqnarray*}
(\partial_\tau - \E)[\Y+s]&=& \partial_\tau s - \li [s] - \q[s] -\M[\Y,s] \\
&=& \lambda\left( \delta\phi^2 + \li[\phi^2] -\lambda \q [\phi^2] +\M[\Y,\phi^2]\right),
\end{eqnarray*}
and once substituting $\li[\phi^2], \q [\phi^2] $, and $\M[\Y,\phi^2] $, we have
\begin{equation}\label{eq:barrier}
(\partial_\tau - \E)[\Y+s] 
=\lambda\left( (\delta + 3 \lambda-1) \phi^2 + 2 (2-\sqrt{2}) \phi-3\phi^ {-1} (2 - \sqrt{2})\right).
\end{equation}

Let $y(\phi,\tau)$ be a solution coming from Ricci flow. Suppose that $y(\phi,0)$ belongs to the class $\mc{C}$ of initial data.

Choose $\lambda(0)=1/5$ and let $\delta$ be a positive number smaller than $10^{-6}$ to be fixed. Then, by \eqref{eq:barrier} the function
$$y_1(\phi,\tau)=\Y(\phi)-\frac{1}{5}e^{-\delta\tau}\phi^2$$
satisfies $(\partial_\tau-\E)[y_1]<0$ for all times $\tau>0$, i.e., $y_1(\phi,\tau)$ is a subsolution to our evolution problem.

 By the comparison principle (Appendix \ref{ap6.2}), if a solution $y$ of $(\partial_\tau-\E)y=0$ initially starts above $y_1$, then $y$ will stay above $y_1$ for all later times, as long as the boundary data behave as such.

A general solution $y$ that initially belongs to class $\mc{C}$ satisfies the assumption $y(\phi,0)>y_1(\phi,0)$ and has the boundary conditions:
 $$y(1)=y((b_0-3a_0)e^\tau+3)=0.$$
 It is  clear that $y(1)=0>y_1(1)$ for all times. Also:
$$y_1(\phi,\tau)=\Y - \frac{1}{5} \exp(-\delta\tau)\phi^2 < \phi -\frac{1}{5} \exp(-\delta\tau)\phi^2< 0 $$
 if $\phi > 5 \exp (\delta\tau)$ (here we are using that $\Y(\phi)$ is always below $\phi$). Because $(b_0-3a_0)e^\tau+3>  5 \exp (\delta \tau)$, if $\delta$ is chosen to be small enough, then we have $y((b_0-3a_0)e^\tau+3)=0>y_1((b_0-3a_0)e^\tau+3)$. Thus, the subsolution $y_1$ stays below $y$ at the boundary, and therefore everywhere, for all later times $\tau>0$. In particular, because $\lambda$ is decreasing in magnitude, this implies $y(\phi,\tau)$ belongs to $\mc{C}$ for all later times.

Furthermore, for the same initial data $y(\phi,0)\in\mc{C}$, we choose
$$y_2(\phi,\tau)=\Y + \lambda_0e^{-\tau/2}\phi^2  $$
where $\lambda_0$ is big enough so that $y_2(\phi,0)=\Y+ \lambda_0\phi^2>y(\phi,0)$. Moreover, by equation \eqref{eq:barrier} we will have that $(\partial_\tau-\E)[y_2]>0$ for all $\tau>0$. Moreover, one can check that the boundary data of $y_2$ stay above those of $y$ for all times. Thus, again using comparison, $y_2$ stays above $y$ for all later times.

Hence we have proved that
\begin{equation}
y_1(\phi,\tau)\leq y(\phi,\tau) \leq y_2(\phi,\tau),
\end{equation}
and since on compact intervals for $\phi$  we have $y_1(\phi,\tau)\nearrow \Y(\phi)$ and $y_2(\phi,\tau)\searrow \Y(\phi)$ uniformly, we must have $y(\phi,\tau)$ converging to $\Y(\phi)$ uniformly  on compact subsets and the proposition thus is proved.
 \end{proof}
 Proposition \ref{conv1} gives us uniform $C^0$-convergence for $y(\phi)=\phi_\rho$ on compact subsets. We next prove that one actually has uniform $C^{1,1}$-convergence. For that we will use the type I blow up of the scalar curvature proved by Song in \cite{So}, i.e.,  there exists a constant $C>0$ such that 
 $$-C\leq R_{g(t)}\leq\frac{C}{T-t}.$$
 In particular, for the dilated flow $\overline{g}(\tau)$ this implies
$$-Ce^{-\tau}\leq R_{\overline{g}(\tau)}\leq C,$$ 
and since $ R_{\overline{g}}(\tau)=\frac{4}{\phi}\left(1-y_\phi\right)-2y_{\phi\phi}$, we must have $y_{\phi\phi}$ is uniformly bounded by  Lemma \ref{boundyphi}.
 

\begin{proposition}\label{c2}
For any initial data $y(\phi,0)$ in the class $\mc{C}$, one has that $y(\phi,\tau)$ converges uniformly in the $C^{1,1}~$topology on compact subsets to $\Y(\phi)$ as $\tau\nearrow\infty$.
\end{proposition}
\begin{proof}
Since we have uniform $C^2$ bounds for $y$ on compact intervals, the $C^{1,1}$-convergence follows from the following standard argument and Proposition \ref{conv1}. Since the spatial derivative $y_{\phi\phi}$ is uniformly bounded in time, for any sequence of times, $y_\phi$ will  converge uniformly up to subsequence. Moreover, since the convergence is uniform, the limit of $y_\phi$ along any such subsequence must be the spatial derivative  $\Y_\phi$ of the stationary state $\Y$. Because the latter does not depend on the subsequence, we have that $y_\phi$ converges uniformly in time to the derivative $\Y_\phi$.  Since $y$ converges uniformly in $C^1$ while $y_{\phi\phi}$ is uniformly bounded, the proposition is then proved.
\end{proof}

\begin{remark}\label{curv}
For the reader's convenience, we point out the relation between the Riemann curvature and derivatives of $y$.  By the rotational symmetry of $\bar{g}(\tau)$, we reduce our analysis to a point of the form $(z^1,z^2)=(\xi,0)$, and use \eqref{riemann} to compute
\begin{eqnarray}\label{riemann2}
|\Rm(\bar{g}(\tau))|&\leq& 2|R_{1\bar{1}1\bar{1}}|+ 2|R_{2\bar{2}2\bar{2}}|+2|R_{1\bar{1}2\bar{2}}|\notag\\
&\leq&2\left| \frac{1}{\phi_\rho^2}\left( -\phi_{\rho\rho\rho}+\frac{\phi^2_{\rho\rho}}{\phi_\rho}\right)\right| + 4\left| \frac{-\phi_\rho+\phi}{\phi^2}\right| + 2\left| -\frac{\phi_{\rho\rho}}{\phi\phi_\rho}+\frac{\phi_\rho}{\phi^2}\right| \notag\\
&=&2\left|\frac{1}{\phi_\rho} \left(\frac{\phi_{\rho\rho}}{\phi_\rho}\right)_\rho\right| +\frac{4}{\phi}\left|-\frac{\phi_\rho}{\phi}+1\right| + \frac{2}{\phi}\left| -\frac{\phi_{\rho\rho}}{\phi_\rho}+\frac{\phi_\rho}{\phi}\right|.
\end{eqnarray}
Moreover, since $\bar{g}(\tau)=\frac{1}{T-t}g(t)$, we note that by Song's \cite{So} type I result there exists a uniform constant $C>0$ such that
$$|\Rm(\bar{g}(\tau))|\leq C.$$
Recalling that  $\frac{\phi_{\rho\rho}}{\phi_\rho}=y_\phi$, $\frac{1}{\phi_\rho} \left(\frac{\phi_{\rho\rho}}{\phi_\rho}\right)_\rho=y_{\phi\phi}$, and also that one has $\frac{\phi_\rho}{\phi}$ uniformly bounded by Lemma \ref{phirphi}, we see from the above that type I blow up is what one really needs to establish a second derivative bound for $y$. 
\end{remark}

Combining Proposition \ref{c2} with Song's type I result \cite{So} we have the following:

\begin{theorem}\label{conv2}
Let $g(0)$ be a metric on $M$ belonging to the class $\mc{C}$. Then the Ricci flow \eqref{eqKRF} starting at $g(0)$ will develop a type I singularity along $\Sigma_0$. Moreover, parabolic dilations $\overline{g}(\tau)$ of $g(t)$ as in \eqref{dilated} will converge, mo\-du\-lo diffeomorphisms and uniformly on a corresponding time-dependent neighborhood of the singular region $\Sigma_0$, to one of the FIK solitons.
\end{theorem}
\begin{proof}
We let $\varphi$ be in the class of initial data $\mc{C}$ and consider the dilation va\-ria\-bles $\rho,\tau$ as well as the dilated potential $\phi(\rho,\tau)$. We apply the following gauge-fixing construction: by the implicit function theorem, we can find a smooth function $C(\tau)$ such that if we define $\mu=\rho-C(\tau)$ and then define
$$\upphi (\mu,\tau)= \phi(\mu-C(\tau),\tau),$$
we will have $\upphi(0,\tau)$ constant in time, say equal to 2.  Moreover, $\upphi_\mu$ is related to  $\upphi $ just  as $\phi_\rho$  is related to $\phi$, and thus  we have by Proposition $\ref{c2}$, $y(\upphi)$ converges $C^{1,1}$ uniformly to $\Y(\upphi)$, so that $\upphi $ converges $C^{2,1}$ uniformly (in compact $\upphi $ intervals, so in particular for $-\infty \leq \mu\leq0$) to the unique $\upphi^\textrm{FIK}$
that satisfies $\upphi^\textrm{FIK}(0)=2$. 

Finally, since we know that the singularity along $\Sigma_0$ is type I \cite{So}, the Riemann curvature $\Rm(\bar{g}(\tau))$ will be bounded uniformly in the region  $-\infty\leq \mu\leq0$. Shi's local estimates for higher derivatives of the Riemann curvature under Ricci flow then dictate that one actually has bounds of any higher order and that thus convergence is smooth. 
\end{proof}

A few remarks are now in order.

\begin{remark}\label{remark3} For any fixed $\rho\in \R$, consider the parabolic neighborhood
$$N(\rho)=\{z \in \C^2\,\,|\,\, \rho(z)\leq \rho\}=\{z \in \C^2\,\,|\,\, |z|^2\leq e^{\rho}(T-t)\}.$$
Theorem \ref{conv2} says that for metrics of the class $\mc{C}$, there exist diffeomorphisms $\Psi_\tau$  (corresponding precisely to the $C(\tau)$ change of gauge) such that as $\tau\nearrow \infty$
 $$\Psi_\tau^{\ast}\bar{g}(\tau)\longrightarrow g_{\textrm{FIK}}$$
  uniformly on the neighborhood $\Psi^{-1}_\tau(N(\rho))$. In the next section, we will prove that the diffeomorphisms $\Psi^{-1}_\tau$ correspond to the 1-parameter family of diffeomorphisms by which the FIK solitons move under Ricci flow, i.e., we will prove that for large $\tau$ one has asymptotically

  $$C(\tau)=(\sqrt{2}-1)\tau \pm\, \textrm{constant}.$$
This will conclude the proof of Theorem \ref{thm:main} that parabolic dilations $\bar{g}(\tau)$ of $g(t)$ converge to the flow of an FIK soliton uniformly on any pa\-ra\-bolic neighborhood $N(\rho)$.
\end{remark}

\begin{remark}\label{remarkC}
If we write $\varphi$ in non-logarithmic coordinates as $\varphi(r,t)=f(w,t)$, where $w=e^r$, and expand $\upphi (\mu,\tau)= \phi(\mu-C(\tau),\tau)=e^\tau\varphi(\mu-C(\tau)-\tau,t)$ around $w=0$ we get
\begin{eqnarray*}
\upphi(\mu,\tau)&=&e^{\tau}f(e^{\mu-C(\tau)-\tau},t) \\
&=& 1 + e^\mu e^{-C(\tau)} f_w(0,t)+\cdots
\end{eqnarray*}
\noindent Because $\upphi$ must converge smoothly to FIK for fixed $\mu$, we must have that \footnote{This notation means that there exist constants $0<C_1<C_2$ such that $C_1e^{C(\tau)-\tau}\leq f_w(0,t)\leq C_2e^{C(\tau)-\tau}$ as $\tau \nearrow \infty$, {\it et cetera}.}
$$f_w(0,t)\sim e^{C(\tau)}.$$
\end{remark}

\begin{remark}\label{rscalar}
Finally, we note that by Theorem \ref{conv2} the scalar curvature at the singular region $\Sigma_0$ must blow up like:
\begin{equation}\label{scalar}
R_{g(t)}\Big|_{\Sigma_0}= \dfrac{4-2\sqrt{2}}{T-t} + O\left((T-t)^{\alpha-1}\right),
\end{equation}
for some $\al>0$. In fact, by the barrier argument in Proposition \ref{conv1}, there must exist positive constants $C_0$ and $\delta_0$ such that  $|y-\Y|\leq C_0 e^{-\delta_0\tau}$ holds uniformly in time on a fixed interval of the form $[1,\phi_0]$, and by Theorem \ref{conv2}, the higher derivatives of $y$ are uniformly bounded in time on $[1,\phi_0]$. Moreover, we can use  an interpolation inequality of the form: ($e.g.$, Corollary 7.21 of \cite{L})
\begin{equation}\label{interpolation}
||\partial_\phi(y-\Y)||_p\leq\epsilon||\partial_{\phi\phi}(y-\Y)||_ p+ \frac{C}{\epsilon} ||y-\Y||_ p, 
\end{equation} 
on $[1,\phi_0]$, where $p>1$, $||\cdot||_p$ is the usual (spatial) $L^p$-norm, $\epsilon$ is any positive number, and $C$ is a universal constant that does not depend on $y,\Y,\epsilon,$ or $p$, and we can thus argue that whenever $||\partial_{\phi\phi}(y-\Y)||^{1/2}_ p\neq0 $, by setting $\epsilon=\frac{||y-\Y||^{1/2}_ p}{||\partial_{\phi\phi}(y-\Y)||^{1/2}_ p}$ in \eqref{interpolation}, one has
$$||\partial_\phi(y-\Y)||_p\leq (1+C) ||y-\Y||^{1/2}_ p|\partial_{\phi\phi}(y-\Y)||^{1/2}_ p, $$ 
and if $||\partial_{\phi\phi}(y-\Y)||^{1/2}_ p=0 $, what we want follows directly from \eqref{interpolation}. Furthermore,  because $[1,\phi_0]$ has finite measure, we can send $p$ to infinity and pass the above inequality to the limit:
\begin{equation}\label{interpolation1}
||\partial_\phi(y-\Y)||_\infty\leq (1+C) ||y-\Y||^{1/2}_ \infty|\partial_{\phi\phi}(y-\Y)||^{1/2}_ \infty,
\end{equation} 
and since $||\partial_{\phi\phi}(y-\Y)||^{1/2}_ \infty$ is uniformly bounded in time and $||y-\Y||_\infty\leq C_0 e^{-\delta_0\tau}$ in $[1,\phi_0]$, we find positive constants $C_1,\delta_1$ such that: 
$$||y_\phi-\Y_\phi||_\infty\leq C_1 e^{-\delta_1\tau}.$$
Moreover, by bootstrapping the argument above, we can also find positive constants $\delta_2,C_2$ such that on $[1,\phi_0]$: 
$$||y_{\phi\phi}-\Y_{\phi\phi}||_\infty\leq C_2 e^{-\delta_2\tau}.$$
Finally, since 
$$ (T-t)R_{g(t)}\Big|_{\Sigma_0}=R_{\overline{g}(\tau)}\Big|_{\Sigma_0}=-2y_{\phi\phi}(1,\tau),$$ and $\Y_{\phi\phi}(1)=\sqrt{2}-2$, $e^{-\tau}=T-t$, \eqref{scalar} follows for $\al=\delta_2>0$.
\end{remark}


\section{End of proof of Theorem \ref{thm:main}}


In this section, we finish the proof of Theorem \ref{thm:main}. This is done by using Theorem \ref{conv2} and the remarks following it.

Let $\varphi(r,t)$ be a potential belonging to the class $\mc{C}$, $\phi(\rho,\tau)$ its corres\-ponding dilated potential as in \eqref{dilated}, and $\upphi(\mu,\tau)$ and $C(\tau)$ as in Theorem \ref{conv2}.  By Remark \ref{remark3}, and since (unnormalized) FIK potentials move under Ricci flow by the diffeomorphisms
$$z \longrightarrow e^{-\sqrt{2}\tau/2}z,$$
it is enough to prove that asymptotically for large $\tau$ one has
\begin{equation}\label{blowC}
C(\tau)=(\sqrt{2}-1)\tau \pm\, \textrm{constant}.
\end{equation}

Note that Remark \ref{remarkC} tells us that if we write $\varphi$ in non-logarithmic coordinates as $\varphi(r,t)=f(w,t)$, we must have then 
$$f_w(0,t)\sim e^{C(\tau)}.$$

This allows us  to use estimate \eqref{scalar} on the blow-up of the scalar curvature along $\Sigma_0$ to study $C(\tau)$ for large $\tau$. In fact, recalling the Ricci eigenvalues, 
$$R_{g(t)}\Big|_{\Sigma_0}=2\lambda_1\Big|_{\Sigma_0}(t)+2\lambda_2\Big|_{\Sigma_0}(t),$$
where $\lambda_1\Big|_{\Sigma_0}(t)=\frac{1}{T-t}$, estimate \eqref{scalar} tells us that the eigenvalue $\lambda_2\Big|_{\Sigma_0}(t)$ must blow up like
\begin{equation}\label{sigma1}
\lambda_2\Big|_{\Sigma_0}(t)= \frac{1-\sqrt{2}}{T-t}+O\left((T-t)^{\al-1}\right),
\end{equation}
for some $\al>0$. Moreover, we can compute $\lambda_2\Big|_{\Sigma_0}(t)$ directly from \eqref{eigen} in the coordinate $w$ and find
\begin{equation}\label{sigma2}
\lambda_2\Big|_{\Sigma_0}(t) = -\dfrac{f_{wt}(0,t)}{f_w(0,t)}.
\end{equation}
Integrating \eqref{sigma2} and using \eqref{sigma1} we find that
\begin{equation}\label{mainestimate}
f_w(0,t)\sim(T-t)^{1-\sqrt{2}}=e^{(\sqrt{2}-1)\tau},
\end{equation}
and this gives \eqref{blowC} by Remark  \ref{remarkC}. Theorem \ref{thm:main} is then proved.


\section{The cone of metrics with non-negative Ricci curvature}


In this section we prove Corollary \ref{cor} by constructing a metric on $M$ with strictly positive Ricci curvature and belonging to the class $\mc{C}$. We recall that  by \eqref{eigen2} one has for the eigenvalues of Ricci and $r$ near $+\infty$ that:
\begin{equation*}\begin{array}{rcl}
\lambda_1 & = & \frac{3}{b_0} +O(e^{-r})\\
\lambda_2 & = &\frac{1}{b_0}+\frac{2b_2}{b_1^2} +O(e^{-r})
\end{array}
\end{equation*}
Let $\varphi_{KC}$ be the potential for the Cao-Koiso soliton, which has positive Ricci curvature everywhere. The metric $\varphi^\textrm{KC}$ is not the metric we are looking for only because $b_0=3a_0$. In fact, one can check explicitly that the Cao-Koiso metric is above the barrier as required in Definition \ref{defc}. Thus we can perturb $\varphi^\textrm{KC}$ by a small amount near $r=+\infty$, to obtain a metric  potential $\varphi$ with $b_0>3a_0$. Since the perturbation is only made near $r+\infty$, where the above expansion for the eigenvalues holds, $\varphi$ will still have strictly positive Ricci curvature everywhere, and also belong to the class $\mc{C}$.

\begin{remark} The construction above provides explicit examples of solutions demonstrating the linear instability of the Cao-Koiso soliton that was proved by Hall-Murphy \cite{HM}.

\end{remark}
\appendix
\section{Line bundles over $\CP^1$ }\label{ap6.1}
On the complex projective space $\CP^1$ with projective coordinates $[z_1:z_2]$, let $\varphi_1:\U_1=\{ [z_1:z_2]\in \CP^1:z_1\neq0\}\longrightarrow \C$ and $\varphi_2:\U_2=\{ [z_1:z_2]\in\CP^1 :z_2\neq0\}\longrightarrow \C$ denote its usual charts given by $\varphi_1([z_1:z_2])=z_2/z_1$ and  $\varphi_2([z_1:z_2])=z_1/z_2$.\\
\indent We consider two topologically distinct line bundles over $\CP^1$ denoted by $L$ and $M$ described as follows.  $L$ has the complex line $\C$ as fibers:
$$L=\displaystyle\left[(\U_1\times\C) \sqcup (\U_2\times\C) \right]\,\,\,\text{quotient by}\,\,\,\sim$$
\noindent where $\U_1\times \C \ni ([z_1:z_2];\xi)\sim ([y_1:y_2],\eta)\in \U_2\times\C$ if, and only if, $[z_1:z_2]=[y_1:y_2]$ and $\eta=\left(\frac{y_2}{z_1}\right)\xi$. The manifold $M$ has fibers $\C\cup\{\infty\}$:
$$M=\displaystyle\left[(\U_1\times\CP^1) \sqcup (\U_2\times\CP^1) \right]/\sim$$
\noindent where  $\U_1\times \CP^1 \ni ([z_1:z_2];\xi)\sim ([y_1:y_2],\eta)\in \U_2\times\CP^1$ if, and only if, $[z_1:z_2]=[y_1:y_2]$ and $\eta=\left(\frac{y_2}{z_1}\right)\xi$. 

For our geometric purposes, we think of  $L$ and $M$ in the following alternative manner. On $M$, consider the global sections $\Sigma_0=\{[z_1:z_2]; 0\}$ and $\Sigma_\infty=\{[z_1:z_2]; \infty\}$ and define a map $\Psi: \C^2\backslash\{0\} \longrightarrow \widehat{M}$, where $\widehat{M}=M\backslash (\Sigma_0\cup \Sigma_\infty)$, given as
$$\Psi:(z_1,z_2)\mapsto ([z_1:z_2];z_\al)$$
\noindent if $z_\al\neq0$. Because $([z_1:z_2];z_\al)\sim ([z_1:z_2];z_\be)$ whenever $z_\al\neq0$ and $z_\be\neq0$, $\Psi$ is well defined. Moreover, one can check that $\Psi$ is biholomorphism.  We then think of $M$ as $\Czero$ with one $\CP^{1}$  glued at $0$ (the section $\Sigma_0$) and another at $\infty$ (the section $\Sigma_\infty$) and of $L=\widehat{M}\cup \Sigma_0 $ as $\Czero$ with a $\CP^{1}$  glued at $0$.

\section{Comparison principle}\label{ap6.2} The comparison principle used in Section 3 for equation \eqref{evol} is similar in spirit to Lemma 3 in \cite{ACK}. For the reader's convenience, we outline the proof in what follows. 

Since the evolving function $y$ in Section 3 is non-negative for all times, any point of contact between $y$ and one of the barriers must be a point where the barrier is non-negative. Moreover, one can check that the barriers used have spatial second derivative bounded in time. We can then reduce our analysis to the following:
\begin{proposition}
Let $y^{-}(\phi,\tau)$ and $y^{+}(\phi,\tau)$ be non-negative sub- and super- solutions, respectively, of $(\partial_\tau-\E)[\cdot]$ on the interval $[1,(b_0-3a_0)e^\tau+3]$. Suppose that either $y^+$  or $y^{-}$ satisfy $y_{\phi\phi}<C$ for some constant $C<\infty$ on a compact space-time set $[1,(b_0-3a_0)e^{\overline{\tau}}+3]\times[0,\overline{\tau}]$. Moreover, assume that 
\begin{itemize}
\item[(i)] $y^{+}(\phi,0)>y^{-}(\phi,0)$ in $(1,b_0-3a_0+3)$;
\item[(ii)] $y^{+}(1,\tau)\geq y^{-}(1,\tau)$ and $y^+((b_0-3a_0)e^\tau+3,\tau)\geq y^-((b_0-3a_0)e^\tau+3,\tau)$, for any $\tau\in[0,\overline{\tau}]$.
\end{itemize}
Then, one must have $y^+(\phi,\tau)\geq y^-(\phi,\tau)$ in $(1,(b_0-3a_0)e^{\overline{\tau}}+3)\times [0,\overline{\tau}]$.
\end{proposition}
\begin{proof}
Suppose first that $y_{\phi\phi}^{+}< C$. For some $\lambda>0$ to be chosen later and any $\al>0$, define
\begin{equation}
w=e^{-\lambda\tau}(y^{+}-y^{-})+\al
\end{equation}
Then $w>0$ on the parabolic boundary of our evolution. We will prove $w$ is also positive in the interior of our domain and the lemma will follow by letting $\al \searrow 0$. 

Assuming the contrary, there must be an interior point $\phi_0$ and a first time $\tau_0$ such that $w(\phi_0,\tau_0)=0$. Then $w_\tau(\phi_0,\tau_0)\leq0$, and at $(\phi_0,\tau_0)$:
$$y^{+}= y^{-} - \al e^{\lambda\tau_0},\qquad y^{+}_\phi= y^{-}_\phi,\qquad y^{+}_{\phi\phi} \geq y^{-}_{\phi\phi}.$$
Thus, at that point, we have $0\geq e^{\lambda\tau}w_\tau$ and
\begin{eqnarray*}
e^{\lambda\tau}w_\tau &=& y^{+}_\tau-y^{-}_\tau -\lambda(y^{+}-y^{-})\\
&=& y^{-}(y^{+}_{\phi\phi}-y^{-}_{\phi\phi}) + (y^{-}-y^{+})\left[\lambda - y^{+}_{\phi\phi}-1 + \frac{y^{+}+y^{-}}{\phi^2}\right].
\end{eqnarray*}
But if we use the uniform bound $y^{+}_{\phi\phi}<C$, we have
$$0\geq w_\tau> \al e^{\lambda\tau}\left[\lambda - C-1 + \frac{y^{+}+y^{-}}{\phi^2}\right],$$
and, since $\frac{y^{+}+y^{-}}{\phi^2}\geq0$, this is a contradiction for any $\lambda>C+1$. The result then follows in the case $y^{+}_{\phi\phi}<C$.

To prove the lemma in the case that the subsolution $y^{-}$ satisfy the uniform bound $y^{-}_{\phi\phi}<C$, one uses the fact that at a first interior zero $(\phi_0,\tau_0)$ of $w$ one has:
$$e^{\lambda\tau}w_\tau = y^{+}(y^{+}_{\phi\phi}-y^{-}_{\phi\phi}) + (y^{-}-y^{+})\left[\lambda - y^{-}_{\phi\phi}-1 + \frac{y^{+}+y^{-}}{\phi^2}\right].$$
\end{proof}

\bibliographystyle{amsalpha}

\end{document}